\documentclass{amsart}
\usepackage{amsmath}
\usepackage{amssymb}
\usepackage{amsthm}

\newcommand{\be}{\begin{equation}}
\newcommand{\ee}{\end{equation}}

\DeclareMathOperator{\dist}{dist}
\DeclareMathOperator{\divergence}{div}

\newtheorem{theorem}{Theorem}[section]

\newtheorem{proposition}[theorem]{Proposition}

\theoremstyle{definition}
\newtheorem{definition}[theorem]{Definition}

\theoremstyle{remark}
\newtheorem{remark}[theorem]{Remark}

\numberwithin{equation}{section}
\begin{document}
%%%%%%%%%%%%%%%%%%%%%%%%%%%%%%%%%%

\title[The double phase problem when the lowest exponent is equal to 1]{The double phase Dirichlet problem when the lowest exponent is equal to 1}
\author{Alexandros Matsoukas}
\address{Department of Mathematics\\
School of Applied Mathematical and Physical Sciences\\
National Technical University of Athens\\
Iroon Polytexneiou 9\\
15780 Zografou\\
Greece}
\email{alexmatsoukas@mail.ntua.gr}
\author{Nikos Yannakakis}
\address{Department of Mathematics\\
School of Applied Mathematical and Physical Sciences\\
National Technical University of Athens\\
Iroon Polytexneiou 9\\
15780 Zografou\\
Greece}
\email{nyian@math.ntua.gr}
\commby{}
%%%%%%%%%%%%%%%%%%%%%%%%%%%%%%%%%%

\begin{abstract}
In this paper we prove an existence and uniqueness result for the double phase Dirichlet problem when the lowest exponent is equal to 1. Our solution is a function of bounded variation that simultaneously lies in a suitable weighted Sobolev space and is found as the limit of a sequence of solutions of intermediate double phase Dirichlet problems whose lowest exponent $p$ goes to 1. As a result of that, our approach involves the study of some relevant properties of generalized Orlicz-Sobolev spaces.
\end{abstract}
%%%%%%%%%%%%%%%%%%%%%%%%%%%%%%%%%%

\maketitle
%%%%%%%%%%%%%%%%%%%%%%%%%%%%%%%%%%

\section{Introduction}
Let $\Omega\subseteq\mathbb R^n$ be a bounded domain with Lipschitz boundary, $f\in L^n(\Omega)$ and $a\in L^\infty(\Omega)$ with $a(x)\geq 0$. In this paper we address the matter of existence of solutions for the double phase Dirichlet problem
%%%%%%%%%%%%%%%%%%%%%%%%%%%%%%%%%%%%%%%%%%%%%%%%%%%%%%
%%%%%%%%%%%%%%%%%%%%%%%%%%%%%%%%%%%%%%%%%%%%%%%%%%%%%%
\begin{equation}
\label{1_laplacian}
\left\{
\begin{array}{rcl}
-\divergence(\displaystyle\frac{\nabla u}{|\nabla u|}+a(x)|\nabla u|^{q-2}\nabla u)&=&f\,\text{ in }\Omega\\
u&=&0\,\text{ in }\partial\Omega\,.
\end{array}
\right.
\end{equation}
%%%%%%%%%%%%%%%%%%%%%%%%%%%
%%%%%%%%%%%%%%%%%%%%%%%%%%%

The terminology double phase problem is used in the literature to describe problems as the above, that depend on two competing exponents $1\leq p<q$ and exhibit distinct growth behavior in different parts of $\Omega$, depending on which exponent is dominant there. These different parts of the domain are usually called the p and the q-phase respectively and here they are determined by the function $a$; when $a(x)=0$ we have the p-phase and when $a(x)>0$ the q-phase respectively.

Double phase problems were firstly introduced by Zhikov in \cite{Zhikov} and in recent years they have become the focus of extensive research
(see for example \cite{mingione_1}, \cite{Colasuonno}, \cite{mingione_2}, \cite{harjulehto_2}, \cite{harjulehto_1}, \cite{mingione_3}, \cite{pap_1}, \cite{pap_2}, \cite{radulescu} and the references therein). Nevertheless with the exception of \cite{harjulehto_2}, \cite{harjulehto_1} it seems that the case $p=1$ has been overlooked and hence this paper is a step towards filling this gap.

As is a common theme when dealing with problems that involve the 1-Laplacian the solution to (\ref{1_laplacian}) will be found as the limit for $p\rightarrow 1$ of the solutions $(u_p)$ of the intermediate Dirichlet problems
%%%%%%%%%%%%%%%%%%%%%%%%%%%%%%%%%%%%%%%%%%%
\begin{equation}
\label{doublephase}
\left\{
\begin{array}{rcl}
-\divergence(|\nabla u|^{p-2}\nabla u+a(x)|\nabla u|^{q-2}\nabla u)&=&f\,\text{ in }\Omega\\
u&=&0\,\text{ in }\partial\Omega\,.
\end{array}
\right.
\end{equation}
%%%%%%%%%%%%%%%%%%%%%%%%%%%
The difficulty of giving some sense to $\frac{\nabla u}{|\nabla u|}$ in (\ref{1_laplacian}) is overcomed by using the ideas of Anzellotti from \cite{Anzellotti} and the role of the above quotient will be played by a bounded vector field $z$. Note also that due to the coexistence of the 1 and the weighted q-Laplacian in (\ref{1_laplacian}), it is natural to expect that its solution should lie simultaneously in $BV(\Omega)$ and in some suitable weighted Sobolev space.

Our approach follows the ideas of \cite{Mercaldo_1}, \cite{Mercaldo_2} and our hypotheses are that the weight $a$ is Lipschitz continuous, belongs to the Muckenhoupt class $A_q$, is non-zero on the boundary of $\Omega$ and the exponents $p,q$ satisfy the inequality $\frac{q}{p}<1+\frac{1}{n}$.
%%%%%%%%%%%%%%%%%%%%%%%%%%%
\section{Preliminaries}
\subsection{Generalized Orlicz spaces}
In this part we follow \cite{paprev}. Let $\Omega\subseteq\mathbb R^n$ be a bounded domain with Lipschitz boundary, $1<p<q<\infty$ and $a\in L^\infty(\Omega)$ with $a(x)>0$ a.e. in $\Omega$. The functions
\[\theta_p, \theta_0:\Omega\times\mathbb R_+\rightarrow\mathbb R_+\]
defined by
\[\theta_p(x,t)=t^p+a(x)t^q\text{ and }\theta_0(x,t)=a(x)t^q\,,\]
are uniformly convex, generalized $\Phi$ functions.
Hence when equipped with the so-called Luxemburg norm
\[\|u\|_{\theta_r} = \inf \{ \lambda>0 : \rho_{\theta_r} (\frac{u}{\lambda}) \le 1\}\,,\]
where
\[\rho_{\theta_r}(u)=\int_{\Omega} \theta_r(x,u)\,dx\,,\]
the generalized Orlicz spaces  $L^{\theta_r}(\Omega)$, for $r=p$ and $0$, are reflexive Banach spaces. The same is true for the generalized Orlicz-Sobolev space
\[W^{1,\theta_r}(\Omega)=\left\{u\in L^{\theta_r}(\Omega): \nabla u\in L^{\theta_r}(\Omega;\mathbb R^n)\right\}\]
with norm
\[\|u\|_{1,\theta_r}=\|u\|_{\theta_r} + \|\nabla u\|_{\theta_r}\,.\]
As usual we define
\[W^{1,\theta_r}_0(\Omega)=\overline{C_{0}^{\infty}(\Omega)}^{\|\cdot\|_{1,\theta_r}}\,,\]
%%%%%%%%%%%%%%%%%%%%%%%%%%%%%%%%%%%%%%%%%%%%%%%%%%%%
\begin{remark}
Note that the spaces $L^{\theta_0}(\Omega)$ and $W^{1,\theta_0}(\Omega)$ are actually weighted Lebesgue and Sobolev spaces respectively. For matters of convenience of notation we may sometimes denote the formers by $L^s_a(\Omega)$, with $a$ the weight and $s$ the corresponding exponent.
\end{remark}
%%%%%%%%%%%%%%%%%%%%%%%%%%%%%%%%%%%%%%%%%%%%%%%%%%%%

If $a\in C(\overline{\Omega})$ and is non-zero on $\partial\Omega$ then we can define a trace on $W^{1,\theta_0}(\Omega)$.
%%%%%%%%%%%%%%%%%%%%%%%%%%%%%%%%%%%%%%%%%%%%%%%%
\begin{proposition}
\label{trace}
Let $a\in C(\overline{\Omega})$ with $a\geq 0$ a.e. in $\Omega$, such that $a(x)\neq 0$, for all $x\in \partial\Omega$. Then there exists a bounded linear operator
\[T:W^{1,\theta_0}(\Omega)\rightarrow L^q(\partial\Omega)\]
such that
\[Tu=u|_{\partial \Omega}\,,\text{ for all }u\in C(\overline{\Omega})\cap W^{1,\theta_0}(\Omega)\,.\]
\end{proposition}
%%%%%%%%%%%%%%%%%%%%%%%%%%%%%%%%%%%%%%%%%%%%%%%%%%%%%%%%%%%%%%%%%%%%%%%%%%%%%%%
\begin{proof}
Since $a(x)\neq 0$, for all $x\in \partial\Omega$, there exists $x_0\in\partial\Omega$ such that $a(x)\geq a(x_0)$, for all $x\in \partial\Omega$. Hence, since $a$ is uniformly continuous we have that there exists $r>0$ such that
$a(x)\geq a(x_0)$, for all $x\in \overline{U}$, where
\[U=\left\{x\in\Omega: \dist(x,\partial\Omega)<r\right\}\]
and this implies (see \cite[p. 17]{kufner}) that
\[W^{1,\theta_0}(U)=W^{1,q}(U)\,.\]
Thus, since $U$ has Lipschitz boundary which contains the boundary of $\Omega$, the existence of the trace operator in $W^{1,q}(U)$ implies the existence of $c>0$ such that
\[\int_{\partial\Omega}|u|^q\,d\mathcal{H}^{n-1}\leq \int_{\partial U}|u|^q\,d\mathcal{H}^{n-1}\leq c\|u\|_{W^{1,p}(U)}\,,\text{ for all } u\in C^1(\overline{U})\,,\]
where by $\mathcal{H}^{n-1}$ we denote the $n-1$ dimensional Hausdorff measure.
%%%%%%%%%%%%%%%%%%%%%%%%%%%%%%%%%%%%%%%%%%%%%%%%%%%%%%%%%%%%%%%%%%%%%%%%%%%%%%%%%%%%%%

Hence the linear operator $S$ defined by $Su=u|_{\partial \Omega}$, for $u\in C^1(\overline{U})$ with values in $L^q(\partial\Omega)$ is bounded and (by density) can be extended to the whole of $W^{1,q}(U)$. Note that as in the classical case it is easy to see that
\[Su=u|_{\partial \Omega},\text{ for all }u\in C(\overline{U})\cap W^{1,q}(U)\,.\]
We now define
\[T:W^{1,\theta_0}(\Omega)\rightarrow L^q(\partial\Omega)\,,\]
by
\[Tu=S(u|_U)\,.\]
The operator $T$ is bounded since our hypothesis on $a$ implies that the operator  $u\rightarrow u|_U$, from $W^{1,\theta_0}(\Omega)$ into $W^{1,q}(U)$ is bounded.

Moreover if $u\in C(\overline{\Omega})\cap W^{1,\theta_0}(\Omega)$, then $u|_{\overline{U}}\in C(\overline{U})\cap W^{1,q}(U)$ and hence
\[Tu=S(u|_U)=u|_{\partial\Omega}\,.\]
\end{proof}
%%%%%%%%%%%%%%%%%%%%%%%%%%%%%%%%%%%%%%%%%%%%%%%%%%%%%%%%%%%%%%%%%%%%%%%%%%%%%%%%
\begin{remark}
The hypothesis that $a$ is non-zero on the boundary of $\Omega$ is not restricting with respect to the double-phase nature of our problem, since the set
$$\left\{x\in \Omega: a(x)=0\right\}$$
still determines the p-phase inside $\Omega$.
\end{remark}
%%%%%%%%%%%%%%%%%%%%%%%%%%%%%%%%%%%%%%%%%%%%%%%%%%%%%%%%%%%%%%%%%%%%%%%%%%%%%%%%%
The above allows us to identify $W^{1,\theta_0}_0(\Omega)$ as the space of functions on $\Omega$ with zero boundary values.
%%%%%%%%%%%%%%%%%%%%%%%%%%%%%%5
\begin{proposition}
\label{zerotrace}
Let $a\in C(\overline{\Omega})$ with $a\geq 0$ a.e. in $\Omega$, such that $a(x)\neq 0$, for all $x\in \partial\Omega$. If $T$ is the trace operator defined in Proposition \ref{trace}, then
\[\ker T=W^{1,\theta_0}_0(\Omega)\,.\]
\end{proposition}
\begin{proof}
The proof follows easily from the classical case.
\end{proof}
%%%%%%%%%%%%%%%%%%%%%%%%%%%%%
%%%%%%%%%%%%%%%%%%%%%%%%%%%%%
\begin{remark}
\label{p_zero_trace}
We have that
\[W^{1,\theta_p}(\Omega)\hookrightarrow W^{1,\theta_0}(\Omega)\]
and hence the restriction of $T$ on $W^{1,\theta_p}(\Omega)$ is the trace operator on $W^{1,\theta_p}(\Omega)$, whose kernel is the space $W^{1,\theta_p}_0(\Omega)$.
\end{remark}
%%%%%%%%%%%%%%%%%%%%%%%%%%%%%%%%%%%%%

If $p<n$ then the following Sobolev embedding holds

\[W_0^{1,\theta_p}(\Omega ) \hookrightarrow L^{p^{*}}(\Omega)
\]
where $p^*=\frac{np}{n-p}$. In particular
\[\|u\|_{p^*}\leq S(n,p)\|u\|_{1,\theta_p}\,,\text{ for all }u\in W_0^{1,\theta_p}(\Omega)\,,\]
where $S(n,p)$ is the best constant of the classical Sobolev embedding. Moreover, as it was showed in \cite{talenti}
\begin{equation}
\label{constant}
\lim_{p\to 1}S(n,p)=S(n,1)
\end{equation}
%%%%%%%%%%%%%%%%%%%%%%%%%%%%%%%%%%%%%%%%%%%%%%%%%%%%%%
and this will be useful in the sequel.

With the additional assumption $\frac{q}{p}<1+\frac{1}{n}$  the following Poincar\'e type inequality holds
\[\|u\|_{\theta_p} \leq c \|\nabla u\|_{\theta_p}\,,\text{ for all }u\in W_0^{1,\theta_p}(\Omega)\,,\]
for some $c>0$.
%%%%%%%%%%%%%%%%%%%%%%%%%%%%%%%%%%%%%%%%%%%%%%%%%%%%%%%%%%%%%%%%%%%%%%%%%%%%%%%%%%%%%%%%%%%%%%

We will also need the definition of the Muckenhoupt class $A_q$.
%%%%%%%%%%%%%%%%%%%%%%%%%%%%%%%%%%%%%%%%%%%%%%%%%%%%%%%%%%%%%%%%%%%%%%%%%%%%%%%%%%%%%%%%%%%%
\begin{definition}
A weight $a\in L^\infty(\Omega)$ with $a(x)>0$ a.e. in $\Omega$ belongs to the Muckenhoupt class $A_q$ if
\[\sup_Q\left(\frac{1}{|Q|}\int_Q a(x) dx\right) \left(\frac{1}{|Q|}\int_Q a(x)^{-\frac{1}{q-1}} dx\right)^{q-1} < \infty,\]
where the supremum is taken over all cubes $Q$ with sides parallel to the coordinate axes.
\end{definition}
%%%%%%%%%%%%%%%%%%%%%%%%%%%%%%%%%%%%%%%%%%%%%%%%%%%%%%%%%%%%%%%%%%%%%%%%%%%%%%%%%%%%%%%%%%%%%%%
For more details on generalized Orlicz spaces we refer the interested reader to the book \cite{harjbook}.
%%%%%%%%%%%%%%%%%%%%%%%%%%%%%%%%%%%%%%%%%%%%%%%%%%%%%%%%%%%%%%%%%%%%%%%%%%%%%%%%%%%%%%%%%%%%%%
%%%%%%%%%%%%%%%%%%%%%%%%%%%%%%%%%%%%%%%%%%%%%%%%%%%%%%%%%%%%%%%%%%%%%%%%%%%%%%%%%%%%%%%%%%%%%%
%%%%%%%%%%%%%%%%%%%%%%%%%%%%%%%%%%%%%%%%%%%%%%%%%%%%%%%%%%%%%%%%%%%%%%%%%%%%%%%%%%%%%%%%%%%%%%
\subsection{Functions of bounded variation and Anzellotti's theory}
A function $u \in L^1(\Omega)$ belongs to $BV(\Omega)$ if its distributional derivative $Du$ is a Radon measure.
The total variation of the measure $Du$ is given by
\[|Du|(\Omega)=\sup \{ \langle Du, \phi \rangle : \phi \in C^\infty_0(\Omega), \|\phi \|\le 1\}.\]
When equipped with the norm
\[\|u\|_{BV}=\|u\|_{L^1(\Omega)}+|Du|(\Omega)\,,\]
the space $BV(\Omega)$ becomes a Banach space and possesses the following important compactness property: if $(u_n)$ is a bounded sequence in $BV(\Omega)$ then there exists a subsequence, which for convenience we denote again by $(u_n)$, and a function $u\in BV(\Omega)$ such that
\begin{itemize}
\item $Du_n\rightarrow Du$, $w^\ast$ as measures in $\Omega$ and
\item $u_n\rightarrow u$, in $L^1(\Omega)$.
\end{itemize}
%%%%%%%%%%%%%%%%%%%%%%%%%%%%%%%%%%%%%%%%%%%%%%%%%%%%%%%%%%%%%%%%%%%%%%%%%%%%%%%%%%%%%%%%%%%%%%
%%%%%%%%%%%%%%%%%%%%%%%%%%%%%%%%%%%%%%%%%%%%%%%%%%%%%%%%%%%%%%%%%%%%%%%%%%%%%%%%%%%%%%%%%%%%%%
%%%%%%%%%%%%%%%%%%%%%%%%%%%%%%%%%%%%%%%%%%%%%%%%%%%%%%%%%%%%%%%%%%%%%%%%%%%%%%%%%%%%%%%%%%%%%%

Let $z\in L^{\infty}(\Omega, \mathbb R^n)$ be a vector field whose $\divergence z$ in the sense of distributions belongs to  $L^\infty(\Omega)$, i.e.
\[\langle \divergence z,\phi \rangle=-\int_\Omega z\,\nabla \phi\, dx\,,\text{ for all }\phi \in C_0^\infty(\Omega).\]
If $u\in BV(\Omega)$ one may define the distribution $(z,Du): C_{0}^{\infty}(\Omega) \to \mathbb{R}$ as
%%%%%%%%%%%%%%%%%%%%%%%%%%%%%%%%%%%%%%%%%%%%%%%
\begin{equation}
\label{anzellotti}
\langle (z,Du), \phi \rangle=-\int_\Omega u z \nabla \phi\, dx-\int_\Omega u\phi \divergence z\, dx.
\end{equation}
%%%%%%%%%%%%%%%%%%%%%%%%%%%%%%%%%%%%%%%%%%%%%%

By Anzellotti's theory \cite{Anzellotti} we can define the weak trace on $\partial\Omega$ of the normal component of $z$, denoted by $[z,v]$ for which we have the following generalized Green's formula
\begin{equation}
\label{Green}
\int_\Omega u \divergence z\,dx+\int_\Omega (z,Du)=\int_{\partial\Omega} u [z,v]\,d\mathcal{H}^{n-1}\,.
\end{equation}

%%%%%%%%%%%%%%%%%%%%%%%%%%%%%%%%%%%%%%%%%%%%%%%%%%%%%%
%%%%%%%%%%%%%%%%%%%%%%%%%%%%%%%%%%%%%%%%%%%%%%%%%%%%%%
%%%%%%%%%%%%%%%%%%%%%%%%%%%%%%%%%%%%%%%%%%%%%%%%%%%%%%
\section{Main results}
As usual a function $u\in W^{1,\theta_p}_{0}(\Omega)$ is said to be a weak solution of the double phase Dirichlet problem (\ref{doublephase}) if
%%%%%%%%%%%%%%%%%%%%%%%%
\begin{equation}
\label{eq1}
\int_{\Omega} (|\nabla u|^{p-2}\nabla u+a(x)|\nabla u|^{q-2}\nabla u) \nabla v dx=\int_{\Omega}f v dx,\
\end{equation}
%%%%%%%%%%%%%%%%%%%%%%%%%%%%%%%%%%%%%%%
for all $v \in W^{1,\theta_p}_{0}(\Omega)$.

Our hypotheses on the weight function $a$ and the exponents $p,q$ are
%%%%%%%%%%%%%%%%%%%%%%%%%%%%%%%%%%%%%%%%%%%%%%%%%%%%%%%%%%%%%%%%%%%%%%%%%%%%%%%%%%%%%
\[(H_0): a\in C^{0,1}(\overline{\Omega})\cap A_q\,, a(x)\neq 0\text{ on }\partial\Omega\,,  1<p<q<n\text{ and }\frac{q}{p}<1+\frac{1}{n}\,,\]
%%%%%%%%%%%%%%%%%%%%%%%%%%%%%%%%%%%%%%%%%%%%%%%%%%%%%%%%%%%%%%%%%%%%%%%%%%%%%%%%%%%%%
where by $C^{0,1}(\overline{\Omega})$ we denote the space of Lipschitz functions defined on $\overline{\Omega}$. We have the following.
%%%%%%%%%%%%%%%%%%%%%%%%%%%%%%%%%%%%%%%%%%%%%%%%%%%%%
\begin{theorem}{\cite[Theorem 1.1]{liu}}
\label{liu}
If $(H_0)$ holds and $f\in L^n(\Omega)$, then there exists a unique weak solution of problem (\ref{doublephase}), for every $p>1$.
\end{theorem}
%%%%%%%%%%%%%%%%%%%%%%%%%%%%%%%%%%%%%%%%
%%%%%%%%%%%%%%%%%%%%%%%%%%%%%%%%%%%%%%%%%%%%5
\begin{remark}
\label{estimate}
Let $p^\ast=\frac{np}{n-p}$. If $f\in L^{p_\ast}(\Omega)$ with $p_\ast$ given by $\frac{1}{p_\ast}+\frac{1}{p^\ast}=1$, then an application of H{\"o}lder's inequality gives us
\[ \|f\|_{p_\ast}\le |\Omega|^{1-\frac{1}{p}} \|f\|_n\,.\]
\end{remark}
%%%%%%%%%%%%%%%%%%%%%%%%%%%%%%%%%%%%%%%%%%%%%%%%%%%%%%%%%%
\begin{proposition}
\label{p_limit}
Assume $(H_0)$ holds and $f\in L^n(\Omega)$ satisfies
$$\|f\|_n(S(n,1)+1) <1\,.$$
If $(u_p)$ are the weak solutions of the Dirichlet problem (\ref{doublephase}), then there exist a function $u\in BV(\Omega)\cap W^{1,\theta_0}_0(\Omega)$ and a vector field $z\in L^{\infty}(\Omega ; \mathbb{R}^{n})$, with $\|z\|_{L^\infty(\Omega; \mathbb R^n)}\le 1$ such that, up to subsequences,

%%%%%%%%%%%%%%%%%%%%%%%%%%%%%%%%%%%%%%%%%%%%%%%%%%%
\begin{eqnarray*}
u_p &\to& u\,, \text{ in } L^1(\Omega)\,,\\
\nabla u_p &\overset{w}{\rightharpoonup}&\nabla u\,, \text{ in } L^{\theta_0} (\Omega; \mathbb R^n),\\
|\nabla u|^{p-2}\nabla u &\overset{w}{\rightharpoonup}& z\,, \text{ in }  L^r (\Omega; \mathbb R^n)\,,
\end{eqnarray*}
%%%%%%%%%%%%%%%%%%%%%%%%%%%%%%%%%%%%%%%%%%%%%%%%%
for all $1\le r<\infty$.
\end{proposition}
%%%%%%%%%%%%%%%%%%%%%%%%%%%%%%%%%%%%%%%%%%%%%%%%%%%%%%%%
\begin{proof}
Let
\[\lambda_p=\|\nabla u_p\|_{\theta_p}=\inf\left\{\lambda >0: \int_\Omega (|\frac{\nabla u_p}{\lambda}|^p+a(x)|\frac{\nabla u_p}{\lambda}|^q)\,dx \le 1\right\}\,.\]
Thus
\[\int_\Omega (|\frac{\nabla u_p}{\lambda_p}|^p+a(x)|\frac{\nabla u_p}{\lambda_p}|^q)\,dx=1.\]
Choosing $u_p/\lambda_p$ as a test function in (\ref{eq1}) we get
\begin{equation}\label{weakform}
\frac{1}{\lambda_p}\int_\Omega (|\nabla u_p|^p+a(x)|\nabla u_p|^q)\, dx=\frac{1}{\lambda_p}\int_\Omega f u_p\, dx\,.
\end{equation}
By H{\"o}lder's inequality, the Sobolev embedding, Remark \ref{estimate} and (\ref{constant}) for $p$ close to $1$ we have
%%%%%%%%%%%%%%%%%%%%%%%%%%%%
\begin{eqnarray*}
I=\frac{1}{\lambda_p}\int_\Omega (|\nabla u_p|^p+a(x)|\nabla u_p|^q)\,dx &\le &\frac{1}{\lambda_p}\|f\|_{p_\ast} \|u_p\|_{p^\ast}\\
&\le &\frac{S(n,p)}{\lambda_p} \|f\|_{p_\ast}\|\nabla u_p\|_{\theta_p}\\
&\le &S(n,p) |\Omega|^{1-\frac{1}{p}} \|f\|_n\,,\\
&\le &(S(n,1)+1) \|f\|_n\,.
\end{eqnarray*}
%%%%%%%%%%%%%%%%%%%%%%%%%%%%%%%%%%
To estimate $I$ on the left hand side, we see that
%%%%%%%%%%%%%%%%%%%%%%%%%%%%%
\begin{eqnarray*}
I &=&{\lambda_p}^{p-1} \int_\Omega |\frac{|\nabla u_p|}{\lambda_p}|^p\,dx + {\lambda_p}^{q-1} \int_{\Omega} a(x) |\frac{\nabla u_p}{\lambda_p}|^q\, dx\\
&\ge &\min\{{\lambda_p}^{p-1},{\lambda_p}^{q-1}\} \int_\Omega (|\frac{|\nabla u_p|}{\lambda_p}|^p + a(x) |\frac{\nabla u_p}{\lambda_p}|^q)\, dx\\
&=&\min\{{\lambda_p}^{p-1},{\lambda_p}^{q-1}\}.\
\end{eqnarray*}
%%%%%%%%%%%%%%%%%%%%%%%%%%%%%%%%%%
Hence
\[\min\{{\lambda_p}^{p-1},{\lambda_p}^{q-1}\}<1\,\]
for $p$ close to $1$ and $\lambda_p=\|\nabla u_p\|_{\theta_p}<1$.

Since $L^{\theta_p}(\Omega) \hookrightarrow L^1(\Omega)$ and $L^{\theta_p}(\Omega) \hookrightarrow L^{\theta_0}(\Omega)$ we also have
\[ \int_\Omega  |\nabla u_p| \le C \|\nabla u_p\|_{\theta_p} \le C\]
and
%%%%%%%%%%%%%%%%%%%%%%%%%%%%%%%%%%%%%%
\begin{equation}
\label{q_estimate}
\left( \int_\Omega a(x) |\nabla u_p|^q \right)^{\frac{1}{q}} \le \|\nabla u_p\|_{\theta_p} \le 1\,.
\end{equation}
%%%%%%%%%%%%%%%%%%%%%%%%%%%%%%%%%%%%%%
Thus $(u_p)$ is bounded in $BV(\Omega)\cap W_0^{1,\theta_0}(\Omega)$ and so passing to a subsequence, which for simplicity we denote again by $(u_p)$, we get a function $u\in BV(\Omega)\cap W_0^{1,\theta_0}(\Omega)$ such that
%%%%%%%%%%%%%%%%%%%%%%%%%%%%%%%%%%%%%%%%%%%%%%%%%%%
\begin{eqnarray*}
\nabla u_p &\overset{w^{*}}{\rightharpoonup}& Du\,, \text{ as measures in } \Omega,\\
\nabla u_p &\overset{w}{\rightharpoonup}& \nabla u\,, \text{ in } L^{\theta_0} (\Omega; \mathbb R^n),\\
u_p &\to & u \text{ in }  L^1(\Omega)\,, \text{ and a.e. in }  \Omega\,.
\end{eqnarray*}
%%%%%%%%%%%%%%%%%%%%%%%%%%%%%%%%%%%%%%%%%%%%%%%%%
Let $1\le r <\infty $. Then since $L^{\theta_p}(\Omega) \hookrightarrow L^p(\Omega)$ we obtain
%%%%%%%%%%%%%%%%%%%%%%%%%%
\begin{eqnarray*}
\int_\Omega |\nabla u_p|^{(p-1)r}\,dx &\le &|\Omega|^{1-\frac{(p-1)r}{p}}\left( \int_\Omega |\nabla u_p|^{p}\,dx \right)^{\frac{(p-1)r}{p}}\\
&\le & |\Omega|^{1-\frac{(p-1)r}{p}} \|\nabla u_p\|_{\theta_p}^{(p-1)r}\\
&\le & |\Omega|^{1-\frac{(p-1)r}{p}}\\
&\le &(|\Omega|+1)\,.
\end{eqnarray*}
%%%%%%%%%%%%%%%%%%%%%%%%
Thus, we have
\begin{equation}
\label{bound}
\||\nabla u|^{p-2}\nabla u\|_{L^r(\Omega; \mathbb{R}^{n})} \le (|\Omega|+1)^{\frac{1}{r}}\,,\text{ for all }1\le r < \infty
\end{equation}
%%%%%%%%%%%%%%%%%%%%%%%%%%%%%%%%%%
and hence we may pass to a further subsequence, still denoted by
$$(|\nabla u_p|^{p-2}\nabla u_p)\,,$$
which converges weakly to a vector field $z_r\in L^r(\Omega; \mathbb R^n)$.

Proceeding with a diagonal argument we may extract a common subsequence $(|\nabla u_p|^{p-2}\nabla u_p)$ and a vector field $z \in L^r(\Omega; \mathbb R^n)$, independent of $r$, such that
%%%%%%%%%%%%%%%%%%%%%%%%%%%%%%%%%%%%%%%%%%%%%%%%%%%
\begin{equation*}
|\nabla u_p|^{p-2}\nabla u_p \overset{w}{\rightharpoonup} z \,, \text{ in } L^r (\Omega; \mathbb R^n ),\
\end{equation*}
%%%%%%%%%%%%%%%%%%%%%%%%%%%%%%%%%%%%%%%%%%%%%%%%%
for all $1\le r<\infty$. Using (\ref{bound}) and the lower semicontinuiuty of the norm, we get
\[\|z\|_{L^r(\Omega; \mathbb R^n)}\le (|\Omega|+1)^{\frac{1}{r}},\]
for all $1\le r < \infty$ and hence  $z\in L^{\infty}(\Omega; \mathbb R^n)$. Letting $r\to \infty$ we have
\[\|z\|_{L^\infty(\Omega; \mathbb R^n)} \le 1.\]
\end{proof}
%%%%%%%%%%%%%%%%%%%%%%%%%%%%%%%%%%%%
%%%%%%%%%%%%%%%%%%%%%%%%%%%%%%%%%%%%

The following Meyers-Serrin type theorem will be important in the sequel. Note that our assumption that the weight function belongs to the Muckenhoupt class is equivalent to the uniform boundedness of the convolution operators and thus, as was pointed out in \cite{Zhikov2}, the Meyers-Serrin theorem is valid in the corresponding weighted Sobolev space $ W^{1,\theta_0}(\Omega)$ .

\begin{proposition}
\label{Meyers-Serrin}
Assume that $(H_0)$ holds. If $u\in BV(\Omega)\cap W_0^{1,\theta_{0}}(\Omega)$, then there exists a sequence $(v_n)$ in $W^{1,1}(\Omega)\cap C^\infty(\Omega)$ such that
\begin{eqnarray*}
v_n &\to& u\,, \text{ in } L^1(\Omega)\,,\\
\int |\nabla v_n|\,dx &\to& |Du|(\Omega)\,,\\
\nabla v_n &\to& \nabla u\,, \text{ in } L^{\theta_0}(\Omega).
\end{eqnarray*}
\end{proposition}

\begin{proof}
Let $\varepsilon >0$. We can define an open covering of $\Omega$ given by a family of open sets $(C_i)_{i\in \mathbb{N}^{*}}$  defined as in \cite[Theorem 10.1.2]{Buttazzo} and $(\phi_i)_{i\in \mathbb{N}^{*}}$ a partition of unity subordinate to the covering. We consider also the standard mollifiers $(\rho_{\varepsilon_i})_{i\in \mathbb{N}^{*}}$. As we have already mentioned since $a\in A_q$ the Meyers-Serrin theorem is valid in $W^{\theta_0}(\Omega)$. Thus we may choose $\varepsilon_i>0$ for each $i\in \mathbb{N}^{*}$, such that
%%%%%%%%%%%%%%%%%%%%%%%%%%%%%%%%%%%%%%%%%%%%%%%%%%%%%%
\begin{eqnarray*}
\text{ supp } \left(\rho_{\varepsilon_i} \ast (\phi_i u)\right) &\subset& C_i,\\
\int_\Omega |\rho_{\varepsilon_i} \ast (\phi_i u)-\phi_i u|\,dx &<& \frac{\varepsilon}{2^i},\\
|\int_\Omega |\rho_{\varepsilon_i} \ast (\phi_i Du)|\,dx- \int_\Omega |\phi_1 Du| | &<& \varepsilon,\\
\int_\Omega |\rho_{\varepsilon_i} \ast (u \nabla \phi_i)-u \nabla \phi_i |\,dx &<& \frac{\varepsilon}{2^i},\\
\int_\Omega a(x) |\rho_{\varepsilon_i} \ast ( \nabla (u \phi_i))- \nabla (u\phi_i) |^{q}\,dx &<& \left(\frac{\varepsilon}{2^i}\right)^q\,.
\end{eqnarray*}
%%%%%%%%%%%%%%%%%%%%%%%%%%%%%%%%%%%%%%%%%%%%%%%%%%%%%%
Hence, defining $v_\varepsilon=\sum_{i=1}^\infty\rho_{\varepsilon_i} \ast (\phi_i u)$ and then arguing as in \cite[Theorem 10.1.2]{Buttazzo}, we have that $v_\varepsilon \in C^{\infty}(\Omega)\cap W^{1,1}(\Omega)$ and approximates $u$ in the sense that
%%%%%%%%%%%%%%%%%%%%%%%%%%%%%%%%%%%%%%%%%%%%%%%%%%%%%%%%
\begin{eqnarray*}
\int_\Omega |v_\varepsilon-u|\, dx &<& \varepsilon,\\
|\int_\Omega |Dv_\varepsilon|- \int_\Omega |Du| | &<& 4\varepsilon\,,
\end{eqnarray*}
%%%%%%%%%%%%%%%%%%%%%%%%%%%%%%%%%%%%%%%%%%%%%%%%%%%%%%%%
and
%%%%%%%%%%%%%%%%%%%%%%%%%%%%%%%%%%%%%%%%%%%%%%%%%%%%%%%%
\begin{eqnarray*}
\left( \int_\Omega a(x) |\nabla v_\varepsilon- \nabla u |^q\, dx \right) ^{\frac{1}{q}} < \varepsilon\,.
\end{eqnarray*}
Hence letting $v_n=v_{\frac{1}{n}}$, for all $n\in\mathbb N$, we obtain the required result.
\end{proof}
%%%%%%%%%%%%%%%%%%%%%%%%%%%%%%%%%%%%%%%%%%%%
%%%%%%%%%%%%%%%%%%%%%%%%%%%%%%%%%%%%%%%%%%%
\begin{remark}
\label{BVWTrace}
By \cite[Remark 10.2.1]{Buttazzo} we have that the function $u\in BV(\Omega)$ and the approximating sequence $(v_n)$ in Proposition \ref{Meyers-Serrin} coincide on the boundary of $\Omega$ in the sense of the $BV$-trace. The same is true for the trace in $W_0^{1,\theta_0}(\Omega)$ since $v_\varepsilon-u$ in the above proof  is the strong limit in $W_0^{1,\theta_0}(\Omega)$ of the sequence
%%%%%%%%%%%%%%%%%%%%%%%%%%%%%%%%%%%%%%%%%%%
\begin{equation}
\nonumber
v_{\varepsilon,k}-u_k=\sum_{i=1}^k (\rho_{\varepsilon_i}\ast (u\phi_i)-u\phi_i)
\end{equation}
%%%%%%%%%%%%%%%%%%%%%%%%%%%%%%%%%%%%%%%%%%%
whose trace is 0 and the result follows from the continuity of the trace operator.
\end{remark}
%%%%%%%%%%%%%%%%%%%%%%%%%%%%%%%%%%%%%
%%%%%%%%%%%%%%%%%%%%%%%%%%
As the following Proposition illustrates, the convergence of the gradients $\nabla u_p$ to $\nabla u$ in $L^{\theta_0}(\Omega, \mathbb R^n)$ as $p\rightarrow 1$, is actually strong
%%%%%%%%%%%%%%%%%%%%%%%%%%%%%%%%%%%
%%%%%%%%%%%%%%%%%%%%%%%%%%%%%%%%%%%%%
\begin{proposition}
\label{q_limit}
Under the hypotheses of Proposition \ref{p_limit} we have that
\begin{equation*}
\nabla u_p  \to \nabla u \text{ in }  L^{\theta_0}(\Omega; \mathbb R^n)\,.
\end{equation*}
\end{proposition}
\begin{proof}
We first show that
%%%%%%%%%%%%%%%%%%%%%%
\begin{equation*}
\lim_{p\to 1}\int_\Omega a(x)\left( |\nabla u_p|^{q-2}\nabla u_p-|\nabla u|^{q-2} \nabla u \right) \nabla (u_p-u)\,dx=0\,.
\end{equation*}
%%%%%%%%%%%%%%%%%%%
By Proposition \ref{p_limit}
%%%%%%%%%%%%%%%%
\begin{equation*}
\nabla u_p \overset{w}{\rightharpoonup} \nabla u \text{ in } L^{\theta_0}(\Omega, \mathbb R^n)
\end{equation*}
%%%%%%%%%%%%%%%%%%%%%%%
and hence
%%%%%%%%%%%%%%%%%%%%%%
\[\lim_{p\to 1}\int_{\Omega} a(x) |\nabla u|^{q-2} \nabla u \nabla (u_{p}-u)\,dx=0\,.\]
%%%%%%%%%%%%%%%%%%%
Thus we only need to show that
%%%%%%%%%%%%%%%%%%%%%%
\[\lim_{p\to 1}\int_{\Omega} a(x) |\nabla u_{p}|^{q-2} \nabla u_{p} \nabla (u_{p}-u)\,dx=0\,.\]
%%%%%%%%%%%%%%%%%%%
To this end let $\varepsilon>0$. Then by Proposition \ref{Meyers-Serrin} there exists $v\in W^{1,1}(\Omega)\cap C^\infty(\Omega)$ such that
%%%%%%%%%%%%%%%
\[\int_{\Omega} |f| |u-v|\,dx + | \int_{\Omega} |\nabla v|\,dx-|Du|(\Omega)|<\frac{\varepsilon}{2}\]
and
%%%%%%%%%%%%%%%%%%%%%%%%%%%%%%%%%
\begin{equation}
\label{q_smooth_estimate}
(\int_\Omega a(x)|\nabla u-\nabla v|^q\, dx)^\frac{1}{q}<\frac{\varepsilon}{2}\,.
\end{equation}
%%%%%%%%%%%%%
Using Remarks \ref{p_zero_trace} and \ref{BVWTrace} we may use $u_p-v$ in the weak formulation (\ref{eq1}) and obtain
%%%%%%%%%%%%%%%%%%%%%%%%
\[\int_\Omega (|\nabla u_p|^{p-2}\nabla u_p+a(x)|\nabla u_p|^{q-2}\nabla u_p) \nabla (u_p-v)\, dx=\int_\Omega f (u_p-v)\,dx,\]
%%%%%%%%%%%%%%%%%%%%%%%%%%%%%%%%%%%%%%%%
or equivalently,
%%%%%%%%%%%%%%%%%%%%%%%%
\[\int_\Omega |\nabla u_p|^p\,dx-\int_\Omega |\nabla u_p|^{p-2}\nabla u_p\nabla v\,dx+\int_\Omega a(x)|\nabla u_p|^{q-2}\nabla u_p \nabla (u_p-v) \, dx=\]
\[=\int_\Omega f (u_p-v)\,dx\,.\]
%%%%%%%%%%%%%%%%%%%%%%%%%%%%%%%%%%%%%%%%
By Young's inequality we obtain
\[
\int_\Omega|\nabla u_p| \le \frac{1}{p}\int_{\Omega}|\nabla u_p|^p +\frac{p-1}p|\Omega|\,,
\]
and so we have
%%%%%%%%%%%%%%%%%%%%%%%%
\[
p\int_\Omega |\nabla u_p|\,dx-\int_\Omega |\nabla u_p|^{p-2}\nabla u_p\nabla v\,dx+\int_\Omega a(x)|\nabla u_p|^{q-2}\nabla u_p \nabla (u_p-v) \,dx\]
\[\le\int_\Omega f (u_p-v) \,dx+(p-1)|\Omega|\,.\]
%%%%%%%%%%%%%%%%%%%%%%%%%%%%%%%%%%%%%%%%
Letting $p\to 1$ using the lower semicontinuity of the total variation and the fact that $|\nabla u_p|^{p-2}\nabla u_p \overset{w}{\rightharpoonup} z$ we have
%%%%%%%%%%%%%%%%%%%%
\[|Du|(\Omega)+\limsup_{p\to 1}\int_\Omega a(x)|\nabla u_p|^{q-2}\nabla u_p \nabla (u_p-v) \,dx\le\]
\[\le\int_\Omega f (u-v)\,dx+\int_\Omega z \cdot \nabla v\,dx\,.\]
%%%%%%%%%%%%%%%%%%%%%%%%%%%%%%%%%%%%
Now since $\|z\|_{L^\infty(\Omega; \mathbb R^n)}\le 1$ we get
%%%%%%%%%%%%%%%
\[
|\int_\Omega z\nabla v|\,dx\le \|z\|_{L^\infty(\Omega; \mathbb R^n)} \int_\Omega |\nabla v|\,dx\le \int_\Omega |\nabla v|\,dx\,\]
%%%%%%%%%%%%%%%%%%%%%%%%
and thus we conclude that
%%%%%%%%%%%%%%%%%%%%
\begin{eqnarray}
\nonumber
\limsup_{p\to 1}\int_\Omega a(x)|\nabla u_p|^{q-2}\nabla u_p \nabla (u_p-v) \,dx\\
\label{limsup}
\le \int_\Omega f (u-v)\,dx+\int_\Omega|\nabla v|\,dx-|Du|(\Omega)<\frac{\varepsilon}{2}\,.
\end{eqnarray}`
%%%%%%%%%%%%%%%%%%%%%%%%%%%%%%%%%%%%
Using H{\"o}lder's inequality, (\ref{q_estimate}) and (\ref{q_smooth_estimate}) we have
%%%%%%%%%%%%%%%%%%
\begin{eqnarray*}
&&\int_\Omega a(x)|\nabla u_p|^{q-2}\nabla u_p \nabla (u_p-u) \,dx=\\
&=&\int_\Omega a(x)|\nabla u_p|^{q-2}\nabla u_p \nabla (u_p-v) \,dx +\int_\Omega a(x)|\nabla u_p|^{q-2}\nabla u_p \nabla (v-u) \,dx\\
&\leq&\int_\Omega a(x)|\nabla u_p|^{q-2}\nabla u_p \nabla (u_p-v) \,dx\\
&+&(\int_\Omega a(x)|\nabla u_p|^q \,dx)^\frac{1}{q'}(\int_\Omega a(x)|\nabla v-\nabla u|^q \,dx)^\frac{1}{q}\\
&<&\int_\Omega a(x)|\nabla u_p|^{q-2}\nabla u_p \nabla (u_p-v) \,dx+\frac{\varepsilon}{2}\,,
\end{eqnarray*}
%%%%%%%%%%%%%%%%%%%%%%%
where $\frac{1}{q}+\frac{1}{q'}=1$.
Passing to the limit as $p\to 1$ and using (\ref{limsup}) we get
\[\limsup_{p\to 1}\int_\Omega a(x)|\nabla u_p|^{q-2}\nabla u_p \nabla (u_p-u) \,dx<\varepsilon\,,\]
for every $\varepsilon>0$ and thus
%%%%%%%%%%%%%%%%%%%%%%
\begin{equation}
\label{zero}
\lim_{p\to 1}\int_{\Omega} a(x)\left( |\nabla u_p|^{q-2}\nabla u_p-|\nabla u|^{q-2} \nabla u \right) \nabla (u_p-u)\,dx=0\,
\end{equation}
%%%%%%%%%%%%%%%%%%%%%%%%%%%%%%%%%%%%%%
since by convexity, the integrand is non-negative.

As  $a(x)>0$ a.e. in $\Omega$ we may now apply a similar argument as the one in \cite[Lemma 5]{Boccardo} to get that $\nabla u_{p} \to \nabla u$ pointwise in $\Omega$. Combining this with the fact that $(|\nabla u_{p}|^{q-2} \nabla u_p)$ is bounded in  $L^{q'}_a(\Omega; \mathbb{R}^n)$, for $\frac{1}{q}+\frac{1}{q'}=1$, we conclude that
%%%%%%%%%%%%%%%%%%%%%%%%%%%%%%%%%%%%%%%%
\begin{equation*}
|\nabla u_{p}|^{q-2} \nabla u_p \overset{w}{\rightharpoonup} |\nabla u|^{q-2} \nabla u\,, \text{ in } L^{q'}_a(\Omega; \mathbb{R}^n)\,.
\end{equation*}
%%%%%%%%%%%%%%%%%%%%%%%%%%%%%%%%%%%%%%%
Therefore we get that
%%%%%%%%%%%%%%%%%%%%%%%%%%%%%%%
\begin{eqnarray*}
\lim_{p\to 1} \int_{\Omega} a(x) |\nabla u_{p}|^{q}\,dx&=& \lim_{p\to 1}\int_{\Omega} a(x) |\nabla u|^{q-2} \nabla u \nabla (u_{p}-u)\,dx\\
& + &\lim_{p\to 1}\int_{\Omega} a(x) |\nabla u_p|^{q-2} \nabla u_p \nabla u\,dx\\
& =& \int_{\Omega} a(x) |\nabla u|^{q}\,dx\,.
\end{eqnarray*}
%%%%%%%%%%%%%%%%%%%%%%%%%%%%%%%
By uniform convexity the space $L^{\theta_0}(\Omega,\mathbb R^n)$ possesses the Radon-Riesz property and thus
\[\nabla u_p \to \nabla u\text{ in }L^{\theta_0}(\Omega,\mathbb R^n)\]
as was required.
\end{proof}
%%%%%%%%%%%%%%%%%%%%%%%%%%%%%%%%%%%%%%%%%%%%%%%%%%%%%%%%%%%%%%%%%%%
%%%%%%%%%%%%%%%%%%%%%%%%%%%%%%%%%%%%%%%%%%%%%%%%%%%%%%%%%%%%%%%%%%%
%%%%%%%%%%%%%%%%%%%%%%%%%%%%%%%%%%%%%%%%%%%%%%%%%%%%%%%%%%%%%%%%%%%
We now introduce the notion of solution of the Dirichlet problem (\ref{1_laplacian}).
%%%%%%%%%%%%%%%%%%%%%%%
\begin{definition}\label{solution} A function $u\in BV(\Omega)\cap W_0^{\theta_0}(\Omega)$ is said to be a solution of the Dirichlet problem (\ref{1_laplacian}), if there exists a vector field $z\in L^\infty(\Omega ; \mathbb R^n)$ with $\|z\|_\infty\le 1$ such that
%%%%%%%%%%%%%%%%%%%%%%%%%%%%%%%%%%%%
\begin{eqnarray*}
\int_\Omega z \nabla \phi + a(x)|\nabla u|^{q-2}\nabla u \nabla \phi\,dx &=& \int_\Omega f \phi \,dx, \text{ for all } \phi \in C^\infty_0(\Omega),\\
-\divergence z-\Delta^a_q u &=&f, \text{ as distributions in }  \Omega,\\
(z,Du) &=&|Du|, \text{ as measures in } \Omega\,.
\end{eqnarray*}
%%%%%%%%%%%%%%%%%%%%%%%%%
\end{definition}
%%%%%%%%%%%%%%%%%%%%%
\begin{remark} If $u$ is a solution of the Dirichlet problem (\ref{1_laplacian}) applying Green's formula (\ref{Green}) we have the following weak formulation
\begin{eqnarray*}
\int_{\Omega} |Du| - \int_{\Omega} (z,Dv)+ \int_{\Omega} a(x) |\nabla u|^{q-2} \nabla u \nabla (u-v)dx = \int_{\Omega} f (u-v)dx,\
\end{eqnarray*}
for all $v\in BV(\Omega)\cap W_0^{\theta_0}(\Omega)$.
\end{remark}
%%%%%%%%%%%%%%%%%%%%%%%%%%%%
%%%%%%%%%%%%%%%%%%%%%%%%%%%%%%%%%%%%
We are now ready for our existence result.
%%%%%%%%%%%%%%%%%%%%%%%%%%%%%%%%%%%%%%%%
\begin{theorem}
\label{main}
Assume $(H_0)$ holds and $f\in L^n(\Omega)$ satisfies
$$\|f\|_n(S(n,1)+1) <1\,.$$
Then there exists a solution of problem (\ref{1_laplacian}) in the sense of Definition \ref{solution}.
\end{theorem}
\begin{proof}
Let $(u_p)$ be the solutions of the Dirchlet problems (\ref{doublephase}). Then for every $p$ we have that
\begin{equation}
\label{p_problem}
\int_\Omega|\nabla u_p|^{p-2}\nabla u_p\nabla \phi\,dx+\int_\Omega a(x)|\nabla u_p|^{q-2}\nabla u_p\nabla \phi\,dx=\int_\Omega f\phi\,dx\,,
\end{equation}
for all $\phi\in C^\infty_0(\Omega)$.

Using Propositions \ref{p_limit} and \ref{q_limit} we get that there exists $u\in BV(\Omega)\cap W_0^{\theta_0}(\Omega)$ and a vector field $z\in L^\infty(\Omega ; \mathbb R^n)$ with $\|z\|_{L^\infty(\Omega; \mathbb R^n)} \le 1$, such that passing to the limit as $p\rightarrow 1$ in (\ref{p_problem}) we have
\begin{equation}
\nonumber
\int_\Omega z \nabla \phi\,dx + \int_\Omega a(x)|\nabla u|^{q-2}\nabla u \nabla \phi\,dx = \int_\Omega f \phi \,dx, \text{ for all } \phi \in C^\infty_0(\Omega)
\end{equation}
and hence also
\begin{equation}
\label{distributions}
-\divergence(z+a(x)|\nabla u|^{q-2}\nabla u)=f\text{ as distributions in }  \Omega\,.
\end{equation}

Using Anzelloti's theory we get that
\[|(z,Du)|\leq \|z\|_\infty |Du|\]
and thus since $\|z\|_{L^\infty(\Omega; \mathbb R^n)} \le 1$ we have that
\[(z,Du)\leq |Du|\,,\text{ as measures in }\Omega\,.\]
To prove the reverse inequality let $\phi\in C^\infty_0(\Omega)$ with $\phi(x)\geq 0$. Using $u_p\phi$ as a test function in (\ref{p_problem}) we get that
%%%%%%%%%%%%%%%%%%%%%%%%%%%%%%%%%%%%%%%%%%%%%
\[\int_\Omega|\nabla u_p|^{p-2}\nabla u_p\nabla (u_p\phi)\,dx+\int_\Omega a(x)|\nabla u_p|^{q-2}\nabla u_p\nabla (u_p\phi)\,dx=\int_\Omega fu_p\phi\,dx\]
and hence
%%%%%%%%%%%%%%%%%%%%%%%%%%%%%%%%%%%%%%%%%%%%%%%%%
\begin{eqnarray}
\nonumber
\int_\Omega\phi|\nabla u_p|^p\,dx+\int_\Omega u_p|\nabla u_p|^{p-2}\nabla u_p\nabla \phi\,dx+\int_\Omega a(x)\phi|\nabla u_p|^q\,dx\\
\label{equalityasmeasures}
+\int_\Omega a(x)u_p\phi|\nabla u_p|^{q-2}\nabla u_p\nabla \phi\,dx
=\int_\Omega fu_p \phi\,dx\,.
\end{eqnarray}
%%%%%%%%%%%%%%%%%%%%%%%%%%%%%%%%%%%%%%%%%%%
Moreover by Young's inequality we get that
%%%%%%%%%%%%%%%%%%%%%%%%%%%%%%%%%%%%%%%%%%%%%%%
\begin{equation}
\nonumber
\int_\Omega\phi|\nabla u_p|\,dx\leq\frac{1}{p}\int_\Omega\phi|\nabla u_p|^p\,dx+\frac{p-1}{p}\int_\Omega \phi\,dx\,.
\end{equation}
%%%%%%%%%%%%%%%%%%%%%%%%%%%%%%%%%%%%%%%%%%%%%%
Therefore together with (\ref{equalityasmeasures}) we have
%%%%%%%%%%%%%%%%%%%%%%%%%%%%%%%%%%%%%%%%%%%%%%%%%%
\begin{eqnarray}
\nonumber
p\int_\Omega\phi|\nabla u_p|\,dx+\int_\Omega u_p|\nabla u_p|^{p-2}\nabla u_p\nabla \phi\,dx+\int_\Omega a(x)\phi|\nabla u_p|^q\,dx\\
\label{equalityasmeasures2}
+\int_\Omega a(x)u_p|\nabla u_p|^{q-2}\nabla u_p\nabla \phi\,dx
\leq\int_\Omega fu_p \phi\,dx+(p-1)\int_\Omega \phi\,dx\,.
\end{eqnarray}
%%%%%%%%%%%%%%%%%%%%%%%%%%%%%%%%%%%%%%%%%%%%%%%%%%%%%%%%%
Using Propositions \ref{p_limit}, \ref{q_limit} and the lower semicontinuity of the total variation, passing to the limit in (\ref{equalityasmeasures2})  as $p\rightarrow 1$ we get that
%%%%%%%%%%%%%%%%%%%%%%%%%%%%%%%%%%%%%%%%5
%%%%%%%%%%%%%%%%%%%%%%%%%%%%%%%%%%%%%%%%%%%%%%%%%
\begin{eqnarray}
\nonumber
\int_\Omega\phi|Du|\,dx+\int_\Omega uz\nabla \phi\,dx+\int_\Omega a(x)\phi|\nabla u|^q\,dx\\
\label{ending}
+\int_\Omega a(x)u|\nabla u|^{q-2}\nabla u\nabla \phi\,dx
=\int_\Omega fu\phi\,dx\,.
\end{eqnarray}
%%%%%%%%%%%%%%%%%%%%%%%%%%%%%%%%%%%%%%%%%%%
On the other hand since $u\in W^{1,\theta_0}_0(\Omega)$ by density we have that
%%%%%%%%%%%%%%%%%%%%%%%%%%%%%%%%%%%%%%%%%%%
\begin{eqnarray*}
\nonumber
\langle -\divergence a(x)|\nabla u|^{q-2}\nabla u, u\phi\rangle&=&\int_\Omega a(x)u|\nabla u|^{q-2}\nabla u\nabla\phi\,dx\\
&+&\int_\Omega a(x)\phi|\nabla u|^q\,dx\,.
\end{eqnarray*}
%%%%%%%%%%%%%%%%%%%%%%%%%%%%%%%%%%%%%%%%%%%
Thus by (\ref{distributions}) and also using (\ref{anzellotti}) we get that
%%%%%%%%%%%%%%%%%%%%%%%%%%%%%%%%%%%%%%%%%%%
\begin{eqnarray*}
\int_\Omega fu\phi\,dx&=&\int_\Omega \phi(z,Du)+\int_\Omega uz\nabla\phi\,dx\\
&+&\int_\Omega a(x)u|\nabla u|^{q-2}\nabla u\nabla\phi\,dx+\int_\Omega a(x)\phi|\nabla u|^q\,dx\,.
\end{eqnarray*}
%%%%%%%%%%%%%%%%%%%%%%%%%%%%%%%%%%%%%%%%%%%%%
Combining this with (\ref{ending}) we get
\[\int_\Omega\phi |Du|\leq \int_\Omega \phi(z,Du)\]
and the proof is completed.
\end{proof}
%%%%%%%%%%%%%%%%%%%%%%%%%%%%%%%%%%%%%%%%%%%%%%%%%%%%%%%%%%%%%%%%%%%%%%%%%%%%%%%%%%%
%%%%%%%%%%%%%%%%%%%%%%%%%%%%%%%%%%%%%%%%%%%%%%%%%%%%%%%%%%%%%%%%%%%%%%%%%%%%%%%%%%%
%%%%%%%%%%%%%%%%%%%%%%%%%%%%%%%%%%%%%%%%%%%%%%%%%%%%%%%%%%%%%%%%%%%%%%%%%%%%%%%%%%%
As the following Proposition illustrates the solution found in Theorem \ref{main} is unique.
\begin{proposition}
Under the same hypotheses as of Theorem \ref{main}, the solution of problem (\ref{1_laplacian}) is unique.
\end{proposition}
%%%%%%%%%%%%%%%%%%%%%%%%%%%%%%%%
\begin{proof}
Let $u_1,u_2\in BV(\Omega)\cap W_0^{\theta_0}(\Omega)$ be two solutions of (\ref{1_laplacian}). Thus, there exist two vector fields $z_1,z_2\in L^\infty(\Omega ; \mathbb R^n)$ such that the conditions of Definition \ref{solution} are satisfied. Testing with $u_2$ in the weak formulation of $u_1$ we obtain
%%%%%%%%%%%%%%%%%%%%%%%%%%%%
\begin{equation}
\label{w1}
\int_\Omega |Du_1| - \int_\Omega (z_1,Du_2)+ \int_\Omega a(x) |\nabla u_1|^{q-2} \nabla u_1 \nabla (u_1-u_2)\,dx = \int_\Omega f (u_1-u_2)\,dx.\
\end{equation}
%%%%%%%%%%%%%%%%%%%%%%%%%%%%%%%%%%
Analogously we get,
%%%%%%%%%%%%%%%%%%%%%%%%%%
\begin{equation}
\label{w2}
\int_\Omega |Du_2| - \int_\Omega (z_2,Du_1)+ \int_\Omega a(x) |\nabla u_2|^{q-2} \nabla u_2 \nabla (u_2-u_1)\,dx = \int_\Omega f (u_2-u_1)\,dx.\
\end{equation}
%%%%%%%%%%%%%%%%%%%%%%%%%%
Adding (\ref{w1}) and (\ref{w2}) we have
%%%%%%%%%%%%%%%%%%%%%%%%%%%%%%%%%%%%%
\begin{eqnarray*}
\int_\Omega |Du_1| &+&\int_\Omega |Du_2| -\int_\Omega (z_1,Du_2) -\int_\Omega (z_2,Du_1) +\\
&+ &\int_\Omega a(x) \left(|\nabla u_1|^{q-2} \nabla u_1-|\nabla u_2|^{q-2} \nabla u_2 \right) \nabla (u_1-u_2)\,dx =0.\
\end{eqnarray*}
%%%%%%%%%%%%%%%%%%%%%%%%%%
We also know that
\begin{equation*}
\int_\Omega (z_1,Du_2) \le \int_\Omega |Du_2|\text{ and }
\int_\Omega (z_2,Du_1) \le \int_\Omega |Du_1|
\end{equation*}
%%%%%%%%%%%%%%%%%%%%%%%%%%%%%%%%%%%%%5
since $\|z_1\| \le 1$ and $\|z_2\| \le 1$. Hence
\begin{equation*}
\int_\Omega a(x) \left(|\nabla u_1|^{q-2} \nabla u_1-|\nabla u_2|^{q-2} \nabla u_2 \right) \nabla (u_1-u_2)\,dx \le 0
\end{equation*}
which implies that
\begin{equation*}
\int_\Omega a(x) \left(|\nabla u_1|^{q-2} \nabla u_1-|\nabla u_2|^{q-2} \nabla u_2 \right) \nabla (u_1-u_2)\,dx = 0
\end{equation*}
since, by convexity, the integrand is non-negative. We conclude that
\[\nabla u_1=\nabla u_2\,,\text{a.e. in }\Omega\]
and so Poincar\'e's inequality yields $u_1=u_2$ a.e. in $\Omega$.
\end{proof}
%%%%%%%%%%%%%%%%%%%%%%%%%%%%%%%%%%%%%%%%%%%%%%%%%%%%%%%%%%%%%%%%%%%%%%%%%%%
%%%%%%%%%%%%%%%%%%%%%%%%%%%%%%%%%%%%%%%%%%%%%%%%%%%%%%%%%%%%%%%%%%%%%%%%%%%
%%%%%%%%%%%%%%%%%%%%%%%%%%%%%%%%%%%%%%%%%%%%%%%%%%%%%%%%%%%%%%%%%%%%%%%%%%%

\end{document}